\documentclass[11pt,a4paper]{article}
\usepackage[english]{babel}
\usepackage{amsmath,amssymb,amsthm,mathtools,epsfig,color,graphicx,dsfont,accents,mathrsfs}
\usepackage[latin1]{inputenc}
\usepackage{tikz}
\usepackage{bbm}
  \newtheorem*{theorem*}        {Theorem}
  \newtheorem*{conjecture*}   {Conjecture}
  \newtheorem{theorem}           {Theorem}
  \newtheorem{lemma}              {Lemma}
  \newtheorem*{lemma*}          {Lemma}
  
  \newtheorem{corollary}          {Corollary}
  \newtheorem{proposition}      {Proposition}
  
    \newtheorem{remark}          {Remark}

\definecolor{Red}{cmyk}{0,1,1,0}

\definecolor{Blue}{cmyk}{1,1,0,0}

\newcommand{\ba}{\begin{array}}
\newcommand{\ea}{\end{array}}
\newcommand{\be}{\begin{equation}}
\newcommand{\ee}{\end{equation}}
\newcommand{\ben}{\begin{enumerate}}
\newcommand{\een}{\end{enumerate}}

  \def\exp{\mathop{\textrm{\rm exp}}\nolimits}                    
   \def\d{\mathop{\textrm{\rm d}}\nolimits}                        

\newcommand{\R}{\mathbb{R}}

\newcommand{\Z}{\mathbb{Z}}
\newcommand{\N}{\mathbb{N}}

\newcommand{\cG}{\mathcal{G}}
\newcommand{\cGd}{\mathcal{G}_\bullet}
\newcommand{\cH}{\ensuremath        {\mathcal{H}}}

\newcommand{\bP}{\ensuremath        {\mathbb{P}}}

\newcommand\smallO{
  \mathchoice
    {{\scriptstyle\mathcal{O}}}
    {{\scriptstyle\mathcal{O}}}
    {{\scriptscriptstyle\mathcal{O}}}
    {\scalebox{.7}{$\scriptscriptstyle\mathcal{O}$}}
  }

\newlength{\dhatheight}
\newcommand{\doublehat}[1]{%
    \settoheight{\dhatheight}{\ensuremath{\hat{#1}}}%
    \addtolength{\dhatheight}{-0.35ex}%
    \hat{\vphantom{\rule{1pt}{\dhatheight}}%
    \smash{\hat{#1}}}}

\begin{document}

\title{Local limits of spatial Gibbs random graphs}
\author{
Eric Ossami Endo\thanks{Supported by FAPESP Grants 14/10637-9 and 15/14434-8.} \\
\footnotesize{Department of Applied Mathematics}\\
\footnotesize{\texttt{eric@ime.usp.br}}\\
\footnotesize{Institute of Mathematics and Statistics - IME USP - University of S\~ao Paulo}\\
\footnotesize{Johann Bernoulli Institute for Mathematics and Computer Science - University of Groningen}\\
\\[0.3cm]
Daniel Valesin \\
\footnotesize{Department of Mathematics}\\
\footnotesize{\texttt{D.Rodrigues.Valesin@rug.nl}}\\
\footnotesize{Johann Bernoulli Institute for Mathematics and Computer Science - University of Groningen}
\\
}
\maketitle

\begin{abstract}
We study the spatial Gibbs random graphs introduced in \cite{mv} from the point of view of local convergence. These are random graphs embedded in an ambient space consisting of a line segment, defined through a probability measure that favors graphs of small (graph-theoretic) diameter but penalizes the presence of edges whose extremities are distant in the geometry of the ambient space. In \cite{mv} these graphs were shown to exhibit threshold behavior with respect to the various parameters that define them; this behavior was related to the formation of hierarchical structures of edges organized so as to produce a small diameter. Here we prove that, for certain values of the underlying parameters, the spatial Gibbs graphs may or may not converge locally, in a manner that is compatible with the aforementioned hierarchical structures.
\end{abstract}

{\footnotesize{\bf Keywords:} Random graphs, Gibbs measures, local convergence}

{\footnotesize {\bf Mathematics Subject Classification (2000):} 82C22; 05C80}

\section{Introduction}
In \cite{mv}, the authors introduced and studied a class of random graphs which they called \textit{spatial Gibbs random graphs}. These are random graphs embedded in an ambient space, which in \cite{mv}, was a finite line segment. They are distributed according to a measure that penalizes the presence of edges whose extremities are distant (in terms of the ambient space geometry), but also penalizes graphs with large graph-theoretic diameter. Graphs sampled from this measure may thus be thought of as answering to a compromise between the conflicting requirements of using few long edges and having vertices close to each other in graph distance. The main result of \cite{mv} describes the typical aspect of these graphs as a function of the various parameters that define them. Here, we continue the study of spatial Gibbs random graphs on line segments by considering their local convergence properties. 

Let us explain the definition of spatial Gibbs random graphs and briefly present the results of \cite{mv}. Define the set of \textit{graphs on $\mathbb{Z}$} as
$$\cG = \{g=(V,E):V \subset \mathbb{Z} \text{ and } g \text{ is locally finite}\}.$$
Given a graph $g=(V,E)\in \mathcal{G}$ and two vertices $x,y\in V$, the \emph{distance between $x$ and $y$ in $g$} is the smallest length over all paths in $g$ with endpoints $x$ and $y$. We denote this distance by $\d_g(x,y)$. Let $p \in [1,\infty]$; in case $g = (V, E) \in \cG$ is finite, we define
 $$\mathcal{H}_p(g) =\begin{cases}\left( {\displaystyle \frac{1}{\binom{N}{2}}\sum_{\substack{x,y\in V:\\ x<y}}
}(\d_g(x,y))^p \right)^{\frac{1}{p}}&\text{if } p \in [1,\infty);\\[.8cm] \sup\left\{ \d_g(x,y):\ x,y\in V \right\}&\text{if } p = \infty, \end{cases} $$
 that is, $\mathcal{H}_\infty(g)$ is the graph-theoretic diameter of $g$ and, if $p \in [1,\infty)$, $\mathcal{H}_p(g)$ is a measure of typical distances in $g$.

For each $N\ge 1$ and $\gamma>0$, let $\bP_{N,\gamma}$ be the probability measure on $\cG$ supported on graphs $g=(V,E)$ with $V=[N]:=\{1,\ldots,N\}$ and $E\supset \{\{x,y\}:x,y\in V, |x-y|=1\}$, and so that the events \begin{equation}\label{eq:prob_edges}\{\{x,y\}\in E\}_{x,y\in V, |x-y|>1}\end{equation} are independent, each having probability\begin{equation}\label{eq:prob_p_e}p_{\{x,y\}}=\exp\{-|x-y|^{\gamma}\}.\end{equation} We think of $\bP_{N,\gamma}$ as a ``reference measure'' which we multiply by a Gibbs-type weight, thus obtaining a measure
\begin{equation}\label{def:prob}
\mathbb{P}^{b,p}_{N,\gamma}(g)=\frac{1}{Z^{b,p}_{N,\gamma}}\cdot \exp\{-N^{b}\cdot  \mathcal{H}_p(g)\}\cdot \mathbb{P}_{N,\gamma}(g),\quad g\in \cG,
\end{equation}
where $b \in \mathbb{R}$, $p \in [1,\infty]$ and $Z^{b,p}_{N,\gamma}$ is the normalization constant. In summary, this measure has four parameters: $N \in \N$ is the number of vertices of graphs over which it is supported, $\gamma > 0$ controls the probabilities of the presence of edges in the reference measure, $p \in [1,\infty]$ determines the notion of typical distance that is used, and $b \in \R$ controls the sensitivity of the measure to the value of the typical distance. We denote by $G_N$ a random graph sampled from  $\mathbb{P}_{N,\gamma}$ or  $\mathbb{P}^{b,p}_{N,\gamma}$, depending on the context.

Under the reference measure $\mathbb{P}_{N,\gamma}$, the geometry of the random graph $G_N$ is not too different from that of the line segment on $[N]$. Indeed, using a simple analysis of ``cutpoints'' carried out in \cite{mv}, it is not hard to show that, if $\varepsilon > 0$ is small enough,
\begin{equation}
\label{eq:cut} \lim_{N\to\infty} \mathbb{P}_{N,\gamma}\left(\mathcal{H}_p(G_N) < \varepsilon N\right) = 0.
\end{equation}
This changes drastically by the introduction of the Gibbs weight (at least if the parameter $b$ is large enough). The main result of \cite{mv}, reproduced as Theorem \ref{thm:old} below, is the convergence in probability of the random variable $N^{-1}\log \mathcal{H}_p(G_N)$ under $\bP_{N,\gamma}^{b,p}$, when $\gamma,b,p$ are fixed and $N$ is taken to infinity. The limit is deterministic and given explicitly as a function of the parameters. Not all triples $(\gamma,b,p) \in (0,\infty) \times \mathbb{R} \times [1,\infty]$ are covered by the theorem: the case $\gamma = 1$ is technically challenging and the proof of convergence for certain values of $(b,p)$ in that case is still missing. To identify this set of values, define for each $p \in [1,\infty]$:
$$\mathcal{E}_p = \begin{cases}\bigcup_{k=1}^\infty \left[\frac{k-1}{k},\;\frac{k-1}{k} + \left(0 \vee \frac{2p-(p-1)k}{k(k+1)(k+2p)}  \right)\right] & \text{if } p < \infty;\\[.2cm][0,\frac14] \cup \bigcup_{k=2}^\infty \left\{\frac{k-1}{k} \right\}&\text{if } p = \infty. \end{cases}$$
This set is plotted on Figure \ref{fig:ep}. We are now ready to state:

\begin{figure}[htb]
\begin{center}
\setlength\fboxsep{0pt}
\setlength\fboxrule{0.5pt}
\fbox{\includegraphics[width = \textwidth]{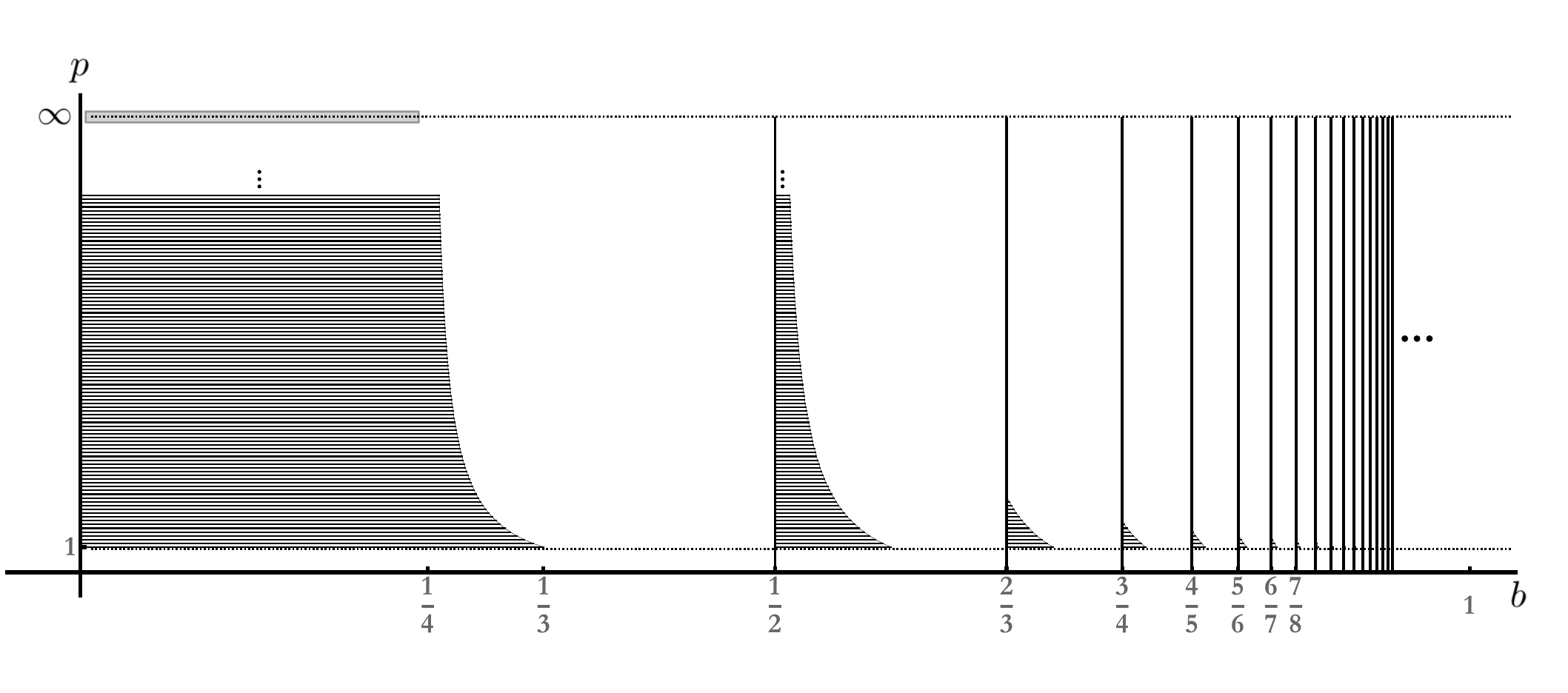}}
\end{center}
\caption{For each $p \in [1,\infty]$, the dark region represents the set $\mathcal{E}_p$, that is, the values of $b$ for which Theorem \ref{thm:old} does not cover the pair $(b,p)$ if $\gamma = 1$. Note that, unless $p = 1$, $\mathcal{E}_p$ only includes finitely many intervals and all numbers of the form $\frac{k-1}{k}$, $k \in \mathbb{N}$. }
\label{fig:ep}
\end{figure}

\begin{theorem}[\cite{mv}]
\label{thm:old} In case
\begin{equation*}
\text{either } \gamma \neq 1,\; p \in [1,\infty],\; b \in \R \qquad \text{or } \gamma = 1,\; p \in [1,\infty],\; b \in \R \backslash \mathcal{E}_p,
\end{equation*}
for any $\varepsilon > 0$,
\begin{equation}
\label{eq:main_eq}
\bP_{N,\gamma}^{b,p}\left(N^{\alpha^*-\varepsilon} < \mathcal{H}_p(G_N) < N^{\alpha^*+ \varepsilon} \right) \xrightarrow{N \to \infty} 1,
\end{equation}
where
\begin{equation*}
\alpha^* = \alpha^*(\gamma, b) = \begin{cases} \left(\frac{1-b}{2-\gamma} \wedge 1\right) \vee 0& \text{if } \gamma \in (0,1);\\[.2cm]
\left(\frac{\gamma - b}{\gamma} \wedge 1\right) \vee 0&\text{if } \gamma > 1;\\[.2cm]
\mathds{1}_{(-\infty,0)}(b) + \sum_{k=1}^\infty \frac{1}{k+1}\cdot \mathds{1}_{\left[\frac{k-1}{k},\frac{k}{k+1}\right)}(b) &\text{if } \gamma = 1. \end{cases}
\end{equation*}
\end{theorem}
Note that the theorem identifies a ``transition window'' for the parameter $b$, given by the intervals $(-1+\gamma,1)$, $(0,\gamma)$ and $(0,1)$ respectively in the cases $\gamma \in (0,1)$, $\gamma > 1$ and $\gamma = 1$. See Figure \ref{fig:al}.

\begin{figure}[htb]
\begin{center}
\setlength\fboxsep{0pt}
\setlength\fboxrule{0.5pt}
\fbox{\includegraphics[width = \textwidth]{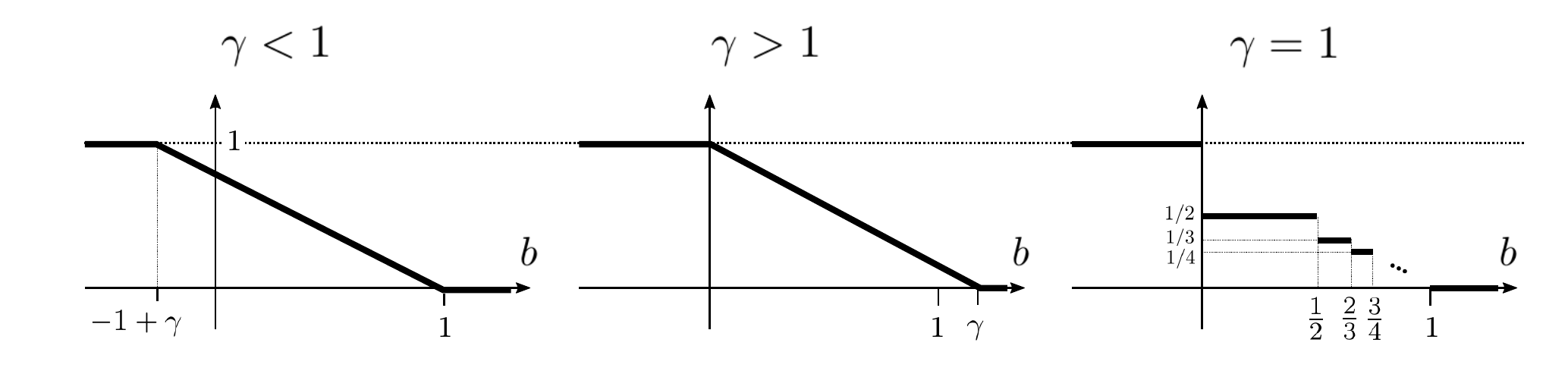}}
\end{center}
\caption{Plot of the function $b\mapsto \alpha^*(\gamma,b)$ of Theorem \ref{thm:old} for the three cases $\gamma \in (0,1)$, $\gamma > 1$ and $\gamma = 1$. }
\label{fig:al}
\end{figure}

In order to motivate our results, it is useful to give a brief exposition of what is involved in the proof of Theorem \ref{thm:old}, carried out in \cite{mv}. Most of the work involves studying the reference measure; specifically, estimating $\mathbb{P}_{N,\gamma}(\mathcal{H}_p(G_N) \leq N^\alpha)$ as $N \to \infty$ for all values of $\alpha \in (0,1)$. Upper and lower bounds whose orders roughly match are obtained for these probabilities. To obtain a lower bound, the authors exhibit a graph $g^\star = g^\star(N,\gamma,\alpha)$ with $\mathcal{H}_p(g^\star)$ close to $N^\alpha$ and use the inequality
$$\mathbb{P}_{N,\gamma}(\mathcal{H}_p(G_N) \leq N^\alpha) \geq \mathbb{P}_{N,\gamma}(g^\star \text{ is a subgraph of } G_N).$$
The definition of $g^\star$ is completely different for the three cases $\gamma \in (0,1)$, $\gamma > 1$ and $\gamma =1$. In order to explain it, let us define, for $N \in \mathbb{N}$ and $\ell \in [N]$, the ``layer'' of edges 
$$\mathscr{E}_{N,\ell} = \{\{1,1+\ell\},\{1+\ell,1+2\ell\},\ldots, \{1+(k-1)\ell,1+k\ell\},\{1+k\ell,N\}\},$$
where $k\geq 0$ is the integer satisfying $1+k\ell < N$, $1+(k+1)\ell \geq N$. Then, $g^\star = ([N],E)$ is defined as follows:
\begin{itemize}
\item in case $\gamma \in (0,1)$, $E = \mathscr{E}_{N,1} \cup \bigcup_{j=0}^i \mathscr{E}_{N,N2^{-j}}$, where $i$ is the smallest integer with $N2^{-i} < N^{1-\alpha}$;
\item in case $\gamma > 1$, $E = \bigcup_{j=0}^i \mathscr{E}_{N,2^j}$, where $i$ is the smallest integer with $2^i > N^\alpha$;
\item in case $\gamma = 1$, $E = \mathscr{E}_{N,1} \cup \bigcup_{j=1}^{i-1} \mathscr{E}_{N,N^{j/i}} $, where $i \geq 2$ is the integer such that $\alpha \in \left(\frac{1}{i},\frac{1}{i-1}\right)$.
\end{itemize}
In all three cases, the layers which constitute $g^\star$ form hierarchical or fractal structures; they are added ``from the top'', ``from the bottom'' and ``from the middle'', respectively, when $\alpha < 1$, $\alpha > 1$ and $\alpha = 1$. See Figures 2 and 5 in \cite{mv} for depictions of these graphs. The proof of the matching upper bound does not quite establish that the mentioned fractal structures are likely to be present in $G_N$. However, it does show that, in agreement with the definition of $g^\star$, the large-deviation event $\{\mathcal{H}_p(G_N) \leq N^\alpha\}$ with $\alpha \in (0,1)$  is most likely to occur due to a coordinated presence of long edges in case $\gamma \leq 1$ and a coordinated presence of short edges in case $\gamma > 1$.

As already mentioned, in this paper we consider the local picture of the spatial Gibbs random graphs. The standard topology for local graph convergence is the one introduced by Benjamini and Schramm in \cite{bs}. This topology involves comparing rooted graphs by asking whether there are graph automorphisms between balls of different radii around the roots. Since here we consider  graphs on $\Z$, the vertices of our graphs are labeled by natural numbers, so it makes sense to modify the Benjamini-Schramm convergence so as to demand that the automorphisms between balls respect the relative positions of the labels. This modification produces a finer topology (that is, if a sequence of rooted graphs converges in the sense to be given below, then it converges in the sense of \cite{bs}). Let us also mention that \cite{bps} also deals with an example of local convergence of rooted graphs endowed with labels or marks.

We now explain the ideas of the previous paragraph precisely. The set of \textit{rooted graphs on $\mathbb{Z}$} is defined by
$$\cGd = \left\{(g,o):\;g\in\mathcal{G},\;o \text{ is a vertex of } g\right\}.$$
For $o,o'\in \mathbb{Z}$, let $\varphi_{o,o'}:\mathbb{Z}\to \mathbb{Z}$ be the translation $$\varphi_{o,o'}(x)=x-o+o'.$$ 
With abuse of notation, for a rooted graph $(g,o)\in \cGd$ with $g=(V,E)$, and $o'\in \mathbb{Z}$, we define $\varphi_{o,o'}(g,o)=((V_{\varphi},E_{\varphi}),o')\in \cGd$ as the rooted graph with
\begin{align*}
&V_{\varphi}=\varphi_{o,o'}(V),\quad E_{\varphi}=\{\{\varphi_{o,o'}(x),\varphi_{o,o'}(y)\}:\{x,y\}\in E\}.
\end{align*}
For $g = (V,E)\in \cG$, $o \in V$, $g' = (V',E')\in \cG$ and $o' \in V'$, we write $(g,o) \simeq (g',o')$ if  $\varphi_{o,o'}(g,o)=(g',o')$.

Given $R>0$ and $(g,o)\in \cGd$ with $g=(V,E)$, a \emph{ball} with center $o$ and radius $R$ in $g$ is the rooted graph $B_{(g,o)}(R)=((V_B,E_B),o)\in \cGd$ of $g$ with
\begin{align*}
&V_B=\{x\in V:\d_g(o,x)\le R\},\quad E_B=\{\{x,y\}\in E:\d_g(o,x)\le R \text{ and }\d_g(o,y)\le R\}.
\end{align*}
A sequence $(g_n,o_n) \in \cGd$ is defined to converge to $(g,o) \in \cGd$ in case
\begin{equation} \forall R \; \exists n_0: \; n \geq n_0 \; \Longrightarrow\; B_{(g_n,o_n)}(R) \simeq B_{(g,o)}(R).
\end{equation}
The associated notion of convergence in distribution is as follows. Given a sequence of random rooted graphs $(G_n,\smallO_n)$ defined under the probability measure $\mu_n$ and a random rooted graph $(G,\smallO)$ defined under the probability measure $\mu$, the sequence $(G_n,\smallO_n)$ converges in distribution to $(G,\smallO)$ if for all $R>0$, and for any deterministic rooted graph $(g,o)\in \cGd$, we have
$$
\lim_{n\to\infty} \mu_n(B_{(G_n,\smallO_n)}(R) \simeq (g,o)) =  \mu(B_{(G,\smallO)}(R) \simeq (g,o)).
$$
Let us give an example that will be useful for the statement of our main result. Let $\mathbb{P}_\gamma$ be the measure on $\mathcal{G}$ supported on graphs $g= (V,E)$ with $V = \mathbb{Z}$ and $E \supset \{\{x,y\}: x,y \in V,\;|x-y| = 1\}$, and so that the events as in \eqref{eq:prob_edges} are independent, with probabilities as in \eqref{eq:prob_p_e}. If $G_N$ is sampled from $\bP_{N,\gamma}$, $G$ is sampled from $\bP_\gamma$ and $a_N$ is a sequence with $1 \ll a_N \ll N$, then it is easy to see that $(G_N, a_N)$ converges in distribution to $(G,0)$.

We now state our main result.
\begin{theorem}\label{thm:main}
Assume $p \in [1,\infty]$ and either of the following conditions hold:
\begin{equation}\label{eq:assumption}[\gamma \in(0,1),\;b \in (-\infty, 1)],\quad [\gamma = 1,\; b \in (-\infty, 1) \backslash \mathcal{E}_p], \quad \text{or } \quad [\gamma > 1,\;  b \in (-\infty, 0)].\end{equation}
Let  $\mathcal{U}_N$ be the uniform measure on $\{1,\ldots,N\}$. Then, $(G_N, \mathcal{O}_N)$ sampled from $\mathbb{P}_{N,\gamma}^{b,p} \otimes \mathcal{U}_N$ converges in distribution to $(G, \mathcal{O})$ sampled from $\mathbb{P}_\gamma \otimes \delta_{\{0\}}$.
\end{theorem}
Intuitively, this result states that, if one of the three conditions holds, then graphs sampled from $\bP_{N,\gamma}$ and $\bP_{N,\gamma}^{b,p}$ are indistinguishable from the point of view of local convergence; in other words, the presence of the Gibbs weight $\exp\{-N^b \cdot \mathcal{H}_p(g)\}$ has no impact on the local picture. Note that, for the regimes $\gamma \in (0,1)$ and $\gamma = 1$, this is compatible with the heuristic explanation we have provided above for the proof of Theorem \ref{thm:old}: in both cases, graphs are most likely to achieve a small diameter by deviating from the reference measure in their long-edge configuration. In the remaining case [$\gamma > 1,\;b < 0$], the idea is that the Gibbs weight is not sufficiently large to cause the random graph to deviate from its local aspect under the reference measure.

Taking this into account, it is not surprising that the study of local convergence is harder for $\gamma > 1,\; b \geq 0$: in that case, short edges do most of the job of reducing the diameter of the graph, so the local picture should be affected by the Gibbs weight. If there is a limiting distribution at all, it would likely differ from $\bP_\gamma$. Our results in this direction are more modest: we show that for a certain subset of the relevant parameters, there is \textit{no} convergence in distribution.

\begin{proposition}\label{prop:main}
For $L > 0$, let $\mathcal{L}_L$ be the set of graphs $g = (V,E) \in \mathcal{G}$ with the property that, if $x,y \in V$ with $0<|x-y|\leq L$, then $\{x,y\} \in E$. For any $L > 0$, $\gamma > 1$, $p < \infty$ and $b > p+1$, then
\begin{equation*}
\bP_{N,\gamma}^{b,p}(G_N \in \mathcal{L}_L) \xrightarrow{N\to\infty} 1.
\end{equation*}
In particular (since $\mathcal{G}$ consist only of locally finite graphs), the sequence $(G_N,\mathcal{O}_N)$ sampled from $\mathbb{P}_{N,\gamma}^{b,p} \otimes \mathcal{U}_N$ does not converge in distribution. 
\end{proposition}
\begin{remark} \begin{enumerate}\item As mentioned after the statement of Theorem \ref{thm:old}, the ``transition window'' for $b$ in case $\gamma > 1$ is the interval $(0,\gamma)$ (regardless of $p$). Hence, if $p+1 < \gamma$, the above proposition shows that there is no local limit even for some values of $b$ within the transition window.
\item For $\gamma > 1$, this leaves open the cases:
$$p = \infty,\; b \geq 0\qquad \text{and}\qquad p \in [1,\infty),\; b \in [0,p+1].$$
We have no guess on whether or not  local convergence occurs for some of these parameter values.
\end{enumerate}\end{remark}

\section{Proof of main result}
\subsection{Truncated balls and proof of Theorem \ref{thm:main}}
In order to prove Theorem \ref{thm:main}, it is enough to fix $\gamma,p,b$ as in \eqref{eq:assumption}, fix $k \in \N$, $(g,o) \in \cGd$, $\varepsilon > 0$, and show that
\begin{equation}
\bP_{N,\gamma}^{b,p}\left(\left|\frac{\#\{i\in[N]:B_{(G_N,i)}(k) \simeq (g,o)\}}{N}  - \mu_\gamma(k,(g,o))\right| > \varepsilon\right) \xrightarrow{N \to \infty} 0,\label{eq:suff}
\end{equation}
where
\begin{equation}\label{eq:def_mu}\mu_\gamma(k,(g,o)) := \bP_\gamma\left(B_{(G,0)}(k) \simeq (g,o)\right).\end{equation}
The natural approach to prove this statement is to first show that, under the reference measure $\bP_{N,\gamma}$, the graph $G_N$ has certain desirable properties with high probability, and then to use this to  draw the desired conclusion about the weighted measure $\bP_{N,\gamma}^{b,p}$. This approach is indeed natural because of the independence properties of the reference measure, which make it easier to study than the weighted measure. However, note that for $i, j \in [N]$, events of the form $\{B_{(G_n,i)}(k) \simeq (g,o)\}$ and $\{B_{(G_n,j)}(k) \simeq (g,o)\}$ are not independent even if $|i-j|$ is large, as both events could be influenced by the presence of long edges with extremities in the vicinities of $i$ and $j$. To deal with this problem, we will introduce \textit{truncated balls} below.

Given an edge $e = \{i,j\}$ of a graph on $\Z$, we define the \textit{length} of $e$ as $|e|:=|i-j|$. Given the rooted graph $(g,o)$ and $k,L \in \N$, we define the truncated ball $B^L_{(g,o)}(k)$ as follows. Let $g'$ be the graph obtained from $g$ by removing all edges with length larger than $L$; then, we let $B^L_{(g,o)}(k) = B_{(g',o)}(k)$.

The essential ingredients in our proof of \eqref{eq:suff} are given in the following result.
\begin{proposition}\label{prop:ingredients}
Fix $\gamma, p, b$ as in \eqref{eq:assumption}.\begin{enumerate} \item For any $k, L > 0$, $(g,o) \in \cGd$ and $\varepsilon > 0$,
\begin{equation}
\mathbb{P}_{N,\gamma}^{b,p}\left( \left|\frac{\#\{i \in [N] : B^L_{(G_N,i)}(k)\simeq (g,o)\}}{N} - \mu_{\gamma}^L(k,(g,o)) \right|> \varepsilon \right) \xrightarrow{N \to \infty} 0,
\label{eq:ing1}\end{equation}
where \begin{equation}\label{eq:def_mu_trunc}\mu_\gamma^L(k,(g,o))=\bP_{\gamma}\left(B^L_{(G,0)}(k) \simeq (g,o) \right). \end{equation}
\item For any $\varepsilon > 0$ there exists $L > 0$ such that
\begin{equation}\label{eq:ing2}
\bP_{N,\gamma}^{b,p}\left( \text{$G_N$ has more than $\varepsilon N$ edges with length larger than $L$}\right) \xrightarrow{N \to \infty} 0.
\end{equation}
\end{enumerate}
\end{proposition}
Let us show how Proposition \ref{prop:ingredients} gives the proof of Theorem \ref{thm:main}; the proof of Proposition \ref{prop:ingredients} will be given afterwards.
\begin{proof}[Proof of Theorem \ref{thm:main}.] Fix $\gamma,p, b$ as in \eqref{eq:assumption}. Also fix $k \in \N$, $(g,o) \in \cGd$ and $\varepsilon > 0$. As already observed, the statement of the theorem will follow once we establish \eqref{eq:suff}. 

Let $\mu_\gamma(k,(g,o))$ be as in \eqref{eq:def_mu} and, for each $L > 0$, let $\mu^L_\gamma(k,(g,o))$ be as in \eqref{eq:def_mu_trunc}. Using the fact that $\bP_\gamma$ is supported on locally finite graphs, it is easy to verify that $\lim_{L \to \infty} \mu^L_\gamma(k,(g,o)) = \mu_\gamma(k,(g,o))$. We can thus choose $L$ large enough that
$$|\mu^L_\gamma(k,(g,o)) - \mu_\gamma(k,(g,o))| < \varepsilon /2.$$
It is then sufficient to prove that
\begin{equation}\label{eq:suff1}
\bP_{N,\gamma}^{b,p}\left(\left|\frac{\#\{i : B_{(G_N,i)}^L(k) \simeq (g,o)\}}{N}  - \mu_\gamma^L(k,(g,o))\right| > \frac{\varepsilon}{4} \right) \xrightarrow{N \to \infty} 0
\end{equation}
and 
\begin{equation}\label{eq:suff2}
\bP_{N,\gamma}^{b,p}\left(\left|\frac{\#\{i : B_{(G_N,i)}(k) \simeq (g,o)\}}{N} - \frac{\#\{i : B^L_{(G_N,i)}(k) \simeq (g,o)\}}{N} \right| > \frac{\varepsilon}{4} \right) \xrightarrow{N \to \infty} 0.
\end{equation}
The convergence \eqref{eq:suff1} is given directly by \eqref{eq:ing1}. For \eqref{eq:suff2}, first observe that, if $L$ is large,
$$\{i \in [N]: B_{(G_N,i)}(k) \simeq (g,o)\} \subseteq \{i \in [N]: B_{(G_N,i)}^L(k) \simeq (g,o)\}.$$
Moreover, if $i_0 \in [N]$ belongs to the set on the right-hand side but not to the set on the left-hand side, then there exists a vertex $x$ of $g$ such that $i_0+x-o$ is an extremity of some edge $e$ of $G_N$ with $|e| > L$. Using these observations, we obtain
\begin{align*}
&|\#\{i\in [N]: B_{(G_N,i)}(k) \simeq (g,o)\} - \#\{i\in [N]: B^L_{(G_N,i)}(k) \simeq (g,o)\}|\\
&\hspace{3cm} \leq 2\cdot \#\{\text{vertices of $g$}\} \cdot \# \{ \text{edges of $G_N$ with length larger than $L$}\}.
\end{align*}
Hence, \eqref{eq:suff2} follows from \eqref{eq:ing2}, completing the proof.

\end{proof}

\subsection{Estimates from \cite{mv}}
We now import some estimates that we will need from \cite{mv} (Lemmas \ref{lem:VM} and \ref{prop:VM} below) and state and prove a consequence of them (Corollary \ref{cor}). 
\begin{lemma}[\cite{mv}]\label{lem:VM}
\begin{enumerate}\item
Assume $\gamma \in (0,1)$. For each $\alpha \in (0,1)$ and $N$ large enough there exists a graph $\hat{g}_{\alpha,N}$ such that
\begin{align*}&\mathcal{H}_p(\hat{g}_{\alpha,N}) \leq N^\alpha\text{ for all }p \in [1,\infty] \;\text{ and } \\&\hspace{1.5cm}\bP_{N,\gamma}(\hat{g}_{\alpha,N} \text{ is a subgraph of } G_N) \geq \exp\left\{-cN^{1-\alpha(1-\gamma)}\right\},\end{align*}
where $c > 0$ depends on $\gamma, \alpha$ but not on $N$.
\item Assume $\gamma = 1$. For each $k \in \N$ and $N$ large enough there exists a graph $\doublehat{g}_{k,N}$ such that
\begin{align*}&\mathcal{H}_p(\doublehat{g}_{k,N}) \leq 3(k+1)N^{\frac{1}{k+1}}\text{ for all } p \in [1,\infty] \; \text{ and }\\[.15cm]&\hspace{1.5cm} \bP_{N,1}\left(\doublehat{g}_{k,N} \text{ is a subgraph of } G_N \right) \geq \exp\left\{-kN\right\}.\end{align*}
\end{enumerate}
\end{lemma}

\begin{lemma}[\cite{mv}]\label{prop:VM}
Assume $p \in [1,\infty]$, $k \in \N$ and $\alpha < \frac{1}{k}$. There exists $\delta > 0$ and a function $o_1(N)$ with $o_1(N)/N \xrightarrow{N \to \infty} 0$ such that, for $N$ large enough,
$$\text{if } g = ([N],E) \in \mathcal{G}\text{ with } \mathcal{H}_p(g) \leq N^\alpha,\; \text{ then } \sum_{\substack{e \in E:\\|e| \geq N^\delta}} |e| \geq kN - o_1(N). $$
\end{lemma}

\begin{corollary}\label{cor}
\begin{enumerate} \item Assume $p \in [1,\infty]$ and either of the following conditions hold:
$$[\gamma \in (0,1),\;b \in (-\infty,1)] \quad \text{or}\quad [\gamma \geq 1,\; b < 0].$$
If $E_N$ are events with $\mathbb{P}_{N,\gamma}(E_N)<\exp\{-\beta N\}$ for some $\beta>0$ and $N$ large, then
\begin{equation}\label{cor:1}
\mathbb{P}^{b,p}_{N,\gamma}(E_N)\xrightarrow{N\to\infty} 0.
\end{equation}
\item Assume $\gamma=1$, $p \in [1,\infty]$ and $b\in [0,1)\backslash \mathcal{E}_p$. Then,
\begin{itemize}
\item[2a.] there exists $C>0$ such that, if $E_N$ are events with $\mathbb{P}_{N,\gamma}(E_N)<\exp\{-CN\}$ for all $N$, then
\begin{equation}\label{cor:2}
\mathbb{P}^{b,p}_{N,1}(E_N)\xrightarrow{N\to\infty} 0;
\end{equation}
\item[2b.] if $E_N$ are events such that $\mathbb{P}_{N,\gamma}(E_N)< \exp\{-cN\}$ for some $c>0$, and each $E_N$ only depends on $\{e:|e|\le L\}$ for a fixed $L$, then
\begin{equation}\label{cor:3}
\mathbb{P}^{b,p}_{N,1}(E_N)\xrightarrow{N\to\infty} 0.
\end{equation}
\end{itemize}
\end{enumerate}
\end{corollary}

\begin{proof}
We start assuming that
$$p\in [1,\infty],\; \gamma > 0,\; b < 0.$$
In this case,
\begin{align*}
Z^{b,p}_{N,\gamma} &\ge \sum_{\substack{g \in \mathcal{G}}} \exp\{-N^b\cdot \mathcal{H}_p(g)\} \cdot \mathbb{P}_{N,\gamma}(g) \geq \sum_{\substack{g \in \mathcal{G}}} \exp\{-N^{b+1}\} \cdot \mathbb{P}_{N,\gamma}(g) =  \exp\{-N^{ b+1}\} 
\end{align*}
since $\mathcal{H}_p(G_N) \leq N$ almost surely under $\mathbb{P}_{N,\gamma}$.
Thus, 
\begin{equation*}
\begin{aligned}
\mathbb{P}^{b,p}_{N,\gamma}(E_N) &\le (Z^{b,p}_{N,\gamma})^{-1} \cdot \mathbb{P}_{N,\gamma}(E_N) \le \exp\left\{-\beta N  + N^{b+1} \right\} \xrightarrow{N\to\infty} 0
\end{aligned}
\end{equation*}
since $b < 0$.
This proves part of statement 1 of the corollary; to complete the proof of statement 1, we now assume $$p \in [1,\infty],\;\gamma \in (0,1),\;b\in (-1+\gamma,1).$$ Applying Lemma \ref{lem:VM}, for any $\alpha\in (0,1)$ we obtain the following lower bound for the partition function, for any $\alpha \in (0,1)$:
\begin{align*}
Z^{b,p}_{N,\gamma}&\ge \exp\{-N^b \cdot \mathcal{H}_p(\hat{g}_{\alpha,N})\}\cdot \bP_{N,\gamma}(\hat{g}_{\alpha,N} \text{ is a subgraph of $G_N$})\\
&\ge \exp\{-N^{b+\alpha} -cN^{1-\alpha(1-\gamma)} \}.
\end{align*}
Thus,
\begin{equation*}
\begin{aligned}
\mathbb{P}^{b,p}_{N,\gamma}(E_N) &\le (Z^{b,p}_{N,\gamma})^{-1} \cdot \mathbb{P}_{N,\gamma}(E_N) \le \exp\left\{-\beta N +N^{b+\alpha}+cN^{1-\alpha(1-\gamma)}\right\}.
\end{aligned}
\end{equation*}
Since $b < 1$, setting $\alpha=\frac{1-b}{2-\gamma}$ gives $b+\alpha<1$ and $1-\alpha(1-\gamma)<1$, proving (\ref{cor:1}).

Now suppose $\gamma=1$, $p \in [1,\infty]$ and $b\in [0,1)\backslash \mathcal{E}_p$. Let $k\in \mathbb{N}$ be the unique integer such that
\begin{equation}\label{eq:b}
\frac{k-1}{k}<b<\frac{k}{k+1}.
\end{equation}
From Lemma \ref{lem:VM}, we have
\begin{equation} \begin{split}
Z^{b,p}_{N,1} & \geq \exp\left\{-N^b \cdot \mathcal{H}_p(\doublehat{g}_{k,N}) \right\} \cdot \bP_{N,\gamma} (\text{$\doublehat{g}_{k,N}$ is a subgraph of $G_N$}) \\[.2cm]&\ge \exp\left\{ -3(k+1)N^{b+\frac{1}{k+1}}-kN \right\}=\exp\{-kN+o(N)\},
\end{split}\label{eq:partition_f} \end{equation}
where $o(N)/N\to 0$. The last equality holds by (\ref{eq:b}). Then, for $C=C(b)>k$, we have
$$
\begin{aligned}
\mathbb{P}^{b,p}_{N,\gamma}(E_N) &\le  (Z^{b,p}_{N,\gamma})^{-1}\cdot \mathbb{P}_{N,\gamma}(E_N)\le \exp\{-CN+kN-o(N)\} \xrightarrow{N\to\infty} 0,
\end{aligned}
$$
proving \eqref{cor:2}. To prove (\ref{cor:3}), fix an arbitrary $\frac{1}{k+1}<\alpha < \frac{1}{k}$. Take $\delta > 0$ and $o_1(N)$ as in Lemma \ref{prop:VM}. Define 
$$
\begin{aligned}
B&=\{\cH_p(G_N)\le N^{\alpha}\} \text{ and }\\
C&=\left\{G_N = ([N],E):\sum_{e\in E: |e|\ge N^{\delta}}|e|\ge kN-o_1(N)\right\}.
\end{aligned}
$$
Then,
$$
\mathbb{P}^{b,p}_{N,1}(E_N)=\mathbb{P}^{b,p}_{N,1}(E_N\cap B)+\mathbb{P}^{b,p}_{N,1}(E_N\cap B^c).
$$
From Theorem \ref{thm:old},
$$
\mathbb{P}^{b,p}_{N,1}(E_N\cap B^c) \le \bP^{b,p}_{N,1}(\cH_p(G_N)> N^{\alpha}) \xrightarrow{N\to \infty} 0.
$$
Lemma \ref{prop:VM} claims that $B\subseteq C$. Moreover, since the event $E_N$ depends only on edges with length at most $L$, we can take $N$ large enough so that the events $E_N$ and $C$ are independent under the reference measure $\bP_{N,1}$. Thus,
\begin{equation*}
\begin{aligned}
\bP^{b,p}_{N,1}(E_N\cap B)\le\bP^{b,p}_{N,1}(E_N\cap C)
&\le (Z^{b,p}_{N,1})^{-1}\cdot \bP_{N,1}(E_N)\cdot \bP_{N,1}(C)\\
&\stackrel{\eqref{eq:partition_f}}{\le} \exp\{kN-o(N)\}\cdot \exp\{-cN\}\cdot \bP_{N,1}(C).
\end{aligned}
\end{equation*}
Now, using Chernoff's inequality it can be shown that 
$$\bP_{N,1}(C) < \exp\{-kN + o(N)\};$$
see the proof of Proposition 3.1 in \cite{mv} for the details. This completes the proof of \eqref{cor:2}.
\end{proof}

\subsection{Estimates for the reference measure and proof of Proposition \ref{prop:ingredients}}

We will use the following concentration result for sums of bounded random variables with finite-range dependence. It is a particular case of Theorem 2.1 of \cite{Janson}.
\begin{lemma}\label{lem:janson}
Let $Y_1,\ldots, Y_n$ be random variables such that, for some $m,L > 0$ and for each $i$, $0 \leq Y_i \leq m$ and
$Y_i$ is independent of $\{Y_j: |j-i| > L\}$. Then, letting $X = \sum_{i=1}^n Y_i$, we have
\begin{equation}
\mathbb{P}\left(|X-\mathbb{E}(X)| > t\right) \leq 2 \exp\left\{-\frac{2t^2}{(2L+1)nm^2} \right\}.
\end{equation}
\end{lemma}

We now state and prove two lemmas which give upper bounds to the probabilities of the same events that appear in the two parts of Proposition \ref{prop:ingredients}; however, in these lemmas, the probability measure under  consideration is the reference measure $\bP_{N,\gamma}$ rather than the weighted measure $\bP_{N,\gamma}^{b,p}$.
\begin{lemma}\label{lem:conv}
Let $\gamma > 0$, $k,L > 0$, $(g,o) \in \cGd$ and $\varepsilon > 0$. For $N$ large enough we have 
\begin{equation}\label{eq:prop_conv}
\mathbb{P}_{N,\gamma}\left( \left|\frac{\#\{i \in [N] : B^L_{(G_N,i)}(k)\simeq (g,o)\}}{N} - \mu^L_{\gamma}(k,(g,o)) \right|> \varepsilon \right) < 2\exp\left\{-\frac{\varepsilon^2N}{8kL+2}\right\},
\end{equation}
where $\mu_{\gamma}^L(k,(g,o))$ is as in \eqref{eq:def_mu_trunc}.
\end{lemma}

\begin{proof} Fix $\gamma$, $k, L$, $(g,o)$ and $\varepsilon$. Also let $N > L$. Define
\begin{align*}
&Y_{N,i} = N^{-1}\cdot\mathds{1}\{B^L_{(G,i)}(k) \simeq (g,o) \},\;i \in [N],\\&X_N = \sum_{i=1}^N Y_{N,i},\qquad \nu^L_{N,\gamma}(k,(g,o)) = \mathbb{E}_{N,\gamma}\left(X_N\right).
\end{align*}
If $1+kL < i < N-kL$, then 
$$\mathbb{P}_{N,\gamma}\left(B^L_{(G_N,i)}(k) \simeq (g,o) \right) =\mu^L_{\gamma}(k,(g,o)),$$
so if $N$ is large enough we have
\begin{equation}|\nu_{N,\gamma}^L(k,(g,o))- \mu^L_{\gamma}(k,(g,o))|<\frac{\varepsilon}{2}.\label{eq:prop_conv15}\end{equation}
Moreover, for each $i \in [N]$ we have $0 \leq Y_{N,i}\leq N^{-1}$ and $Y_{N,i}$ is independent of $\{Y_{N,j}:|j-i|>2kL\}$ under $\mathbb{P}_{N,\gamma}$, so Lemma \ref{lem:janson} yields
\begin{equation}\mathbb{P}_{N,\gamma}\left(|X_N - \nu^L_{N,\gamma}(k,(g,o))| > \frac{\varepsilon}{2}\right) \leq 2\exp\left\{-\frac{2\varepsilon^2N^2}{4(4kL+1)N} \right\} = 2\exp\left\{-\frac{\varepsilon^2N}{8kL+2}\right\}.\label{eq:prop_conv2}
\end{equation}
The result now follows from \eqref{eq:prop_conv15} and \eqref{eq:prop_conv2}.
\end{proof}
\begin{lemma}\label{eq:eloslongos}
For every $\gamma>0$, $N>1$ and $\varepsilon>0$ we have, for $L$ large enough,
\begin{equation*}
\mathbb{P}_{N,\gamma}(\text{$G_N$ has more than $\varepsilon N$ edges with length larger than $L$}) \le \exp\left\{-L^{\gamma/8} N\right\}.
\end{equation*}
\end{lemma}

\begin{proof}
Fix $L\ge 1$ and define 
$$\Lambda_{N,L} = \{\{i,j\}: 1 \leq i < i+L < j \leq N\}.$$
By Chernoff's inequality, for $\theta>0$,
$$
\begin{aligned}
&\mathbb{P}_{N,\gamma}(\text{$G_N$ has more than $\varepsilon N$ edges with length larger than $L$})\\
&\hspace{4cm}\leq \exp\{-\theta \varepsilon N\} \prod_{\{i,j\}\in \Lambda_{N,L}}(p_{\{i,j\}}\exp\{\theta\}+1-p_{\{i,j\}});\end{aligned}$$
we remind the reader that $p_{\{i,j\}} = \exp\{-|j-i|^\gamma\}$. The right-hand side is less than
\begin{align}\nonumber
&\exp\{-\theta \varepsilon N\} \prod_{\{i,j\}\in \Lambda_{N,L}}\exp\{p_{\{i,j\}}(\exp\{\theta\}-1)\}\\&\hspace{3cm}\le \exp\left\{-\theta \varepsilon N +(\exp\{\theta\}-1)\sum_{\{i,j\}\in \Lambda_{N,L}}p_{\{i,j\}}\right\}.\label{eq:reff}
\end{align}
We then bound
$$
\sum_{\{i,j\}\in \Lambda_{N,L}}p_{\{i,j\}}=\sum_{k=L+1}^N (N-k+1)\exp\{-k^{\gamma}\} \le N\sum_{k=L+1}^\infty\exp\{-k^{\gamma}\} < N \exp\{-L^{\gamma/2}\}
$$
if $L$ is large enough.

Choosing $\theta=L^{\frac{\gamma}{4}}$, the expression in \eqref{eq:reff} is at most
$$\exp\left\{ - L^{\frac{\gamma}{4}} \varepsilon N + \exp\{L^{\frac{\gamma}{4}}\} \cdot \exp
\{-L^{\frac{\gamma}{2}} \} \cdot N \right\} < \exp\{-L^\frac{\gamma}{8}N\}$$
if $L$ is large enough.
\end{proof}

\begin{proof}[Proof of Proposition \ref{prop:ingredients}.]
The proposition follows immediately from putting together Corollary \ref{cor} and Lemmas \ref{lem:conv} and \ref{eq:eloslongos}.
\end{proof}

\section{No local convergence for $\gamma>1$, $p < \infty$ and $b > p+1$}

\begin{proof}[Proof of Proposition \ref{prop:main}]
Fix $\gamma>1$, $p\in [1,\infty)$, $b>p+1$ and $L>1$. Also fix $N \in \N$. For $g=([N],E) \in \mathcal{L}_L$ and $k \in \N$,  let $M_{L,k}(g)$ be the set of graphs obtained by removing $k$ edges with length at most $L$ from $g$. For every $g' \in M_{L,k}(g)$ we have
$$
(\mathcal{H}_p(g'))^p\ge (\mathcal{H}_p(g))^p+\frac{k}{\binom{N}{2}}\cdot (2^p-1).
$$
Since $ \mathcal{H}_p(g) \in [1,N]$, the mean value theorem gives
$$
  \mathcal{H}_p(g')- \mathcal{H}_p(g) \ge \frac{( \mathcal{H}_p(g'))^p- (\mathcal{H}_p(g))^p}{p\cdot N^{p-1}}.
$$
Thus, there exists $C_L>0$ such that
$$
g' \in M_{L,k}(g) \quad \Longrightarrow \quad \frac{\bP^{b,p}_{N,\gamma}(G_N=g')}{\bP^{b,p}_{N,\gamma}(G_N=g)} \le \exp\left\{-kN^{b-p-1}\cdot\frac{(2^p-1)}{p}\right\}(C_L)^k.
$$
Noting that $\#M_{L,k}(g) \leq \binom{NL}{k}$, we bound
$$
\begin{aligned}
\sum_{k=1}^{NL}\;\sum_{g' \in M_{L,k}(g)}\frac{\bP^{b,p}_{N,\gamma}(G_N=g')}{\bP^{b,p}_{N,\gamma}(G_N=g)} &\le \sum_{k=1}^{NL}\binom{NL}{k}\exp\left\{-kN^{b-p-1}\cdot\frac{(2^p-1)}{p}\right\}(C_L)^k\\
&\le \sum_{k=1}^{\infty}\left(C_L\cdot N\cdot L\cdot \exp\left\{-N^{b-p-1}\cdot\frac{(2^p-1)}{p}\right\} \right)^k
\xrightarrow{N\to\infty}0
\end{aligned}
$$
since $b > p+1$. Thus,
$$
\begin{aligned}
\bP^{b,p}_{N,\gamma}(G_N \notin \mathcal{L}_L) \leq \sum_{g\in \mathcal{L}_L} \bP^{b,p}_{N,\gamma}(G_N=g) \sum_{k=1}^{NL}\;\sum_{g' \in M_{L,k}(g)}\frac{\bP^{b,p}_{N,\gamma}(G_N=g')}{\bP^{b,p}_{N,\gamma}(G_N=g)}\xrightarrow{N\to\infty}0,\\
\end{aligned}
$$
as desired.
\end{proof}

\section*{Acknowledgments}
The authors would like to thank Aernout van Enter and Jean-Christophe Mourrat for helpful discussions.

\end{document}